\documentclass{article}
\usepackage{amsmath}
\usepackage{amssymb}
\usepackage{amsfonts}

\newtheorem{theorem}{Theorem}

\newtheorem{corollary}[theorem]{Corollary}

\newtheorem{lemma}[theorem]{Lemma}

\newenvironment{proof}[1][Proof]{\noindent\textbf{#1.} }{\ \rule{0.5em}{0.5em}}
\include {mak}
\input{tcilatex}

\begin{document}

\title{{\normalsize SUMS OF PRODUCTS OF GENERALIZED FIBONACCI AND LUCAS
NUMBERS}}
\author{Hac\`{e}ne Belbachir and Farid Bencherif \\
\ \ \ \ \ \ \ \ \ \ \ \ \ \ \ \ \ \ \ \ \ \ \ \ \ \ \ \ \ \ \ \ \ \ \ \ \ \
\ \ \ \ \ \ \ \ \ \ \ \\
USTHB, Fac. Math., P.B. 32 El Alia, 16111, Algiers, Algeria.\\
{\small \ hbelbachir@usthb.dz or hacenebelbachir@gmail.com}\\
{\small \ fbencherif@usthb.dz or fbencherif@gmail.com}}
\maketitle

\begin{abstract}
In this paper, we establish several \ formulae for sums and alternating sums
of products of generalized Fibonacci and Lucas numbers. In particular, we
recover and extend all results of Z. \v{C}erin \cite[2005]{cer1} and Z. \v{C}%
erin and G. M. Gianella \cite[2006]{cer2}, more easily.
\end{abstract}



\section{Introduction and main result}

Let $p$ and $q$ two integers such that $pq\neq 0$ and $\Delta :=p^{2}-4q\neq
0$. We define sequences of generalized Fibonacci and Lucas numbers $\left(
U_{n}\right) =(U_{n}^{\left( p,q\right) })$ and $\left( V_{n}\right)
=(V_{n}^{\left( p,q\right) }),$ for all $n,$ by induction%
\begin{equation*}
\left\{
\begin{array}{l}
U_{0}=0,\ U_{1}=1,\ U_{n}=pU_{n-1}-qU_{n-2} \\
V_{0}=2,\ V_{1}=p,\ V_{n}=pV_{n-1}-qV_{n-2}%
\end{array}%
\right.
\end{equation*}

Sequences of Fibonacci $\left( F_{n}\right) ,$ Lucas $\left( L_{n}\right) ,$
Pell $\left( P_{n}\right) ,$ Pell-Lucas $\left( Q_{n}\right) ,$ Jacobsthal $%
\left( J_{n}\right) ,\ $Jacobsthal-Lucas $\left( j_{n}\right) $ listed
respectively {\small A000045, A00032, A000129, A002203, A001045, A014551 }in%
{\small \ SLOANE} \cite{slo} are $\left( F_{n},L_{n}\right) =(U_{n}^{\left(
1,-1\right) },V_{n}^{\left( 1,-1\right) }),$ $\left( P_{n},Q_{n}\right)
=(U_{n}^{\left( 2,-1\right) },V_{n}^{\left( 2,-1\right) }),$ $\left(
J_{n},j_{n}\right) =(U_{n}^{\left( 1,-2\right) },V_{n}^{\left( 1,-2\right)
}) $ for $n\geq 0.$

For $r$ and $s$ two integers and for all sequences $\left( X_{m}\right)
_{m\in
\mathbb{Z}
}$\ and $\left( Y_{m}\right) _{m\in
\mathbb{Z}
}$, let%
\begin{equation*}
S_{n}^{\left( r,s\right) }\left( X,Y\right) :=\sum_{i=0}^{n}X_{r+2i}Y_{s+2i}%
\text{ \ and \ }A_{n}^{\left( r,s\right) }\left( X,Y\right)
:=\sum_{i=0}^{n}\left( -1\right) ^{i}X_{r+2i}Y_{s+2i},
\end{equation*}%
for convenience, we also set $S_{n}^{\left( r,s\right) }\left( X\right)
:=S_{n}^{\left( r,s\right) }\left( X,X\right) $ and $A_{n}^{\left(
r,s\right) }\left( X\right) :=A_{n}^{\left( r,s\right) }\left( X,X\right) .$

Sums involving Fibonacci$,$ Lucas$,$ Pell and Pell-Lucas numbers and
generalizations have been studied by several authors, for example, for
trigonometric sums see Melham \cite[1999]{mel} and Belbachir \& Bencherif
\cite[2007]{bel4}, for reciprocal and powers sums see Melham \cite[1999]%
{mel1} and \cite[2000]{mel2}, and for the sum of squares see Long \cite[1986]%
{lon}, \v{C}erin \cite[2005]{cer1} and \v{C}erin \& Gianella \cite[2006]%
{cer2}.

In \cite[2005]{cer1}, \v{C}erin studied $A_{n}^{\left( r,s\right) }\left(
L\right) $ for $s=r$ and $s=r+1$ when $r$ is odd, and in \cite[2006]{cer2},
\v{C}erin and Gianella considered $S_{n}^{\left( r,s\right) }\left( Q\right)
$\ and $A_{n}^{\left( r,s\right) }\left( Q\right) $ for $s=r$ and $s=r+1$
when $r$ is even.

Recently, \v{C}erin \cite[2007]{cer3} studied the sums of squares and
products of Jacobsthal numbers by establishing identities for $S_{n}^{\left(
r,s\right) }\left( J\right) ,$\ and\ $A_{n}^{\left( r,s\right) }\left(
J\right) ,$ for $s=r$ and $s=r+1$ when $r$ is even$.$ This case corresponds
to $\left( p,q\right) =\left( 1,-2\right) .$

Our purpose is to give simplified expressions for the sums $S_{n}^{\left(
r,s\right) }\left( U\right) ,$\ $S_{n}^{\left( r,s\right) }\left( V\right) ,$%
\ $A_{n}^{\left( r,s\right) }\left( U\right) $ and $A_{n}^{\left( r,s\right)
}\left( V\right) .$ In all what follows, we suppose $q=\pm 1$ (which gives $%
V_{2}\neq 0,$ $U_{2}\neq 0$ and $U_{4}\neq 0$)$.$

For $n\in
\mathbb{Z}
$, let us define the sequences $\left( a_{n}\right) ,$ $\left( b_{n}\right)
, $ $\left( c_{n}\right) ,$ $\left( d_{n}\right) $ and$\ \left( e_{n}\right)
$ by the relations%
\begin{equation*}
a_{n}=\frac{U_{2n}}{U_{2}},\ b_{n}=\frac{d_{n+1}-1}{p^{2}\Delta },\ c_{n}=%
\frac{U_{4n+4}}{U_{4}},\ d_{n}=\frac{V_{4n+2}}{V_{2}}\text{, }e_{n}=p\left(
d_{n}-1\right) .
\end{equation*}

These sequences, depending on $p$ and $q$ satisfy the recurrence relations

\begin{center}
$%
\begin{array}{lll}
a_{-1}=-1,\ \  & a_{0}=0,\ \  & a_{n}=V_{2}a_{n-1}-a_{n-2}, \\
b_{-1}=0, & b_{0}=1, & b_{n}=V_{4}b_{n-1}-b_{n-2}+1, \\
c_{-1}=0, & c_{0}=1, & c_{n}=V_{4}c_{n-1}-c_{n-2}, \\
d_{-1}=1, & d_{0}=1, & d_{n}=V_{4}d_{n-1}-d_{n-2}, \\
e_{-1}=0, & e_{0}=0, & e_{n}=V_{4}e_{n-1}-e_{n-2}+p^{3}\left( p^{2}-4\right)
.%
\end{array}%
$
\end{center}

For $\left( p,q\right) =\left( 1,-1\right) ,$ we have, for $n\geq 0,$ $%
\left( {\small U}_{n}{\small ,V}_{n}\right) =\left( {\small F}_{n}{\small ,L}%
_{n}\right) $ and one gets $\left( a_{n}\right) =\left( {\small 0,1,3,8,21}%
,\ldots \right) $, $\left( b_{n}\right) =\left( {\small 1,8,56,385,2640,}%
\ldots \right) $, $\left( c_{n}\right) =\left( {\small 1,7,48,329,2255}%
,\ldots \right) $ and $\left( d_{n}\right) =\left( {\small 1,6,41,281,1926,}%
\ldots \right) $ listed in {\small SLOANE} respectively as {\small A001906,
A092521, A004187, A049685}.

For $\left( p,q\right) =\left( 2,-1\right) ,$ we have, for $n\geq 0,$ $%
\left( {\small U}_{n}{\small ,V}_{n}\right) =\left( {\small P}_{n}{\small ,Q}%
_{n}\right) $ and one gets $\left( a_{n}\right) =\left( {\small 0,1,6,35,}%
\ldots \right) $, $\left( b_{n}\right) =\left( {\small 1,35,1190,40426}%
,\ldots \right) $, $\left( c_{n}\right) =\left( {\small 1,34,1155,39236}%
,\ldots \right) $ and $\left( d_{n}\right) =\left( {\small 1,33,1121,38081}%
,\ldots \right) $ listed in {\small SLOANE} respectively as {\small A001109,
A029546, A029547, A077420}.

\ \ \ \ \ \ \ \ \ \ \ \ \ \ \ \ \ \

We give now, for $\varepsilon =\left( 1+\left( -1\right) ^{n}\right) /2$,
the main result of the paper

\begin{theorem}
\label{AA}For all integers$\ r,\ s$ and\ $n\geq 0,$ we have%
\begin{eqnarray*}
S_{n}^{\left( r,s\right) }\left( U\right)
&=&\sum_{i=0}^{n}U_{r+2i}U_{s+2i}=p^{-1}\Delta ^{-1}\left[
U_{4n+r+s+2}-U_{r+s-2}\right] -\left( n+1\right) \Delta ^{-1}q^{r}V_{s-r}, \\
S_{n}^{\left( r,s\right) }\left( V\right)
&=&\sum_{i=0}^{n}V_{r+2i}V_{s+2i}=p^{-1}\left[ U_{4n+r+s+2}-U_{r+s-2}\right]
+\left( n+1\right) q^{r}V_{s-r}, \\
S_{n}^{\left( r,s\right) }\left( U,V\right)
&=&\sum_{i=0}^{n}U_{r+2i}V_{s+2i}=p^{-1}\Delta ^{-1}\left[
V_{4n+r+s+2}-V_{r+s-2}\right] -\left( n+1\right) \Delta ^{-1}q^{r}U_{s-r}, \\
A_{n}^{\left( r,s\right) }\left( U\right) &=&\sum_{i=0}^{n}\left( -1\right)
^{i}U_{r+2i}U_{s+2i}=\Delta ^{-1}V_{2}^{-1}\left[ V_{r+s-2}+\left( -1\right)
^{n}V_{4n+r+s+2}\right] -\varepsilon \Delta ^{-1}q^{r}V_{s-r}, \\
A_{n}^{\left( r,s\right) }\left( V\right) &=&\sum_{i=0}^{n}\left( -1\right)
^{i}V_{r+2i}V_{s+2i}=V_{2}^{-1}\left[ V_{r+s-2}+\left( -1\right)
^{n}V_{4n+r+s+2}\right] +\varepsilon q^{r}V_{s-r}, \\
A_{n}^{\left( r,s\right) }\left( U,V\right) &=&\sum_{i=0}^{n}\left(
-1\right) ^{i}U_{r+2i}V_{s+2i}=V_{2}^{-1}\left[ U_{r+s-2}+\left( -1\right)
^{n}U_{4n+r+s+2}\right] -\varepsilon q^{r}U_{s-r}.
\end{eqnarray*}
\end{theorem}

\begin{corollary}
\label{c1}For all integers$\ r,\ s$ and\ $n\geq 0,$ we have%
\begin{eqnarray}
\Delta S_{n}^{\left( r,s\right) }\left( U\right) &=&a_{n+1}V_{2n+r+s}-\left(
n+1\right) q^{r}V_{s-r},  \label{3.95} \\
S_{n}^{\left( r,s\right) }\left( V\right) &=&a_{n+1}V_{2n+r+s}+\left(
n+1\right) q^{r}V_{s-r},  \label{3.96} \\
S_{n}^{\left( r,s\right) }\left( U,V\right) &=&a_{n+1}U_{2n+r+s}-\left(
n+1\right) q^{r}U_{s-r}, \\
\Delta A_{n}^{\left( r,s\right) }\left( U\right) &=&\left\{
\begin{array}{ll}
d_{m}V_{4m+r+s}-q^{r}V_{s-r} & \text{\ \ if }n=2m \\
-p\Delta c_{m}U_{4m+r+s+2} & \text{\ \ if }n=2m+1%
\end{array}%
\right. ,  \label{3.97} \\
A_{n}^{\left( r,s\right) }\left( V\right) &=&\left\{
\begin{array}{ll}
d_{m}V_{4m+r+s}+q^{r}V_{s-r} & \text{\ \ if }n=2m \\
-p\Delta c_{m}U_{4m+r+s+2} & \text{\ \ if }n=2m+1%
\end{array}%
\right. ,  \label{3.99} \\
A_{n}^{\left( r,s\right) }\left( U,V\right) &=&\left\{
\begin{array}{ll}
d_{m}U_{4m+r+s}-q^{r}U_{s-r} & \text{\ \ if }n=2m \\
-pc_{m}V_{4m+r+s+2} & \text{\ \ if }n=2m+1%
\end{array}%
\right. .
\end{eqnarray}
\end{corollary}

\begin{corollary}
\label{c2}For all integers$\ r,\ s,$ $t$ and\ $n\geq 0,$ we have%
\begin{equation}
S_{n}^{\left( s,s\right) }\left( U\right) -q^{s-r}S_{n}^{\left( r,r\right)
}\left( U\right) =\Delta ^{-1}\left( S_{n}^{\left( s,s\right) }\left(
V\right) -q^{s-r}S_{n}^{\left( r,r\right) }\left( V\right) \right)
=a_{n+1}U_{s-r}U_{2n+r+s+t},  \label{3.11}
\end{equation}%
\begin{equation}
S_{n}^{\left( s,s+t\right) }\left( V\right) +\Delta q^{s-r}S_{n}^{\left(
r,r+t\right) }\left( U\right) =a_{n+1}V_{s-r}V_{2n+r+s+t}.  \label{3.13}
\end{equation}
\end{corollary}

\section{Proof of the main result}

We shall use the following Lemmas

\begin{lemma}
\label{lemma a}For all integers $n,\ m$ and $h,$ we have

$%
\begin{array}{ll}
\mathit{1}.\ \ U_{-n}=-q^{-n}U_{n}, & \ \ \mathit{2}.\ \ V_{-n}=q^{-n}V_{n},
\\
\mathit{3}.\ \ \Delta U_{n}U_{m}=V_{n+m}-q^{m}V_{n-m}, & \ \ \mathit{4}.\ \
V_{n}V_{m}=V_{n+m}+q^{m}V_{n-m}, \\
\mathit{5}.\ \ U_{n}V_{m}=U_{n+m}+q^{m}U_{n-m}, & \ \ \mathit{6}.\ \
V_{n}U_{m}=U_{n+m}-q^{m}U_{n-m}, \\
\mathit{7}.\ \ U_{n}U_{m+h}-U_{n+h}U_{m}=q^{m}U_{h}U_{n-m},\ \ \ \ \ \ \ \ \
\ \  & \ \ \mathit{8}.\ \ V_{n}V_{m+h}-V_{n+h}V_{m}=-q^{m}\Delta
U_{h}U_{n-m}, \\
\mathit{9}.\ \ V_{n}V_{m+h}-\Delta U_{n+h}U_{m}=q^{m}V_{h}V_{n-m}, & \,%
\mathit{10}.\ \ U_{n}V_{m+h}-U_{n+h}V_{m}=-q^{m}U_{h}V_{n-m}.%
\end{array}%
$
\end{lemma}

\begin{lemma}
\label{lemma c}For all integers $r$ and $n\geq 0,$ we have

1. $\ \Delta U_{2}\sum_{i=0}^{n}U_{r+4i}=V_{4n+r+2}-V_{r-2}=\Delta
U_{2n+r}U_{2n+2},$

2. \ $U_{2}\sum_{i=0}^{n}V_{r+4i}=U_{4n+r+2}-U_{r-2}=V_{2n+r}U_{2n+2},$

3. \ $V_{2}\sum_{i=0}^{n}\left( -1\right) ^{i}U_{r+4i}=\left( -1\right)
^{n}U_{4n+r+2}+U_{r-2}=\left\{
\begin{array}{ll}
U_{2n+r}V_{2n+2} & \text{\ if }n\text{ is even} \\
-V_{2n+r}U_{2n+2} & \text{\ if }n\text{ is odd}%
\end{array}%
\right. ,$

4. \ $V_{2}\sum_{i=0}^{n}\left( -1\right) ^{i}V_{r+4i}=\left( -1\right)
^{n}V_{4n+r+2}+V_{r-2}=\left\{
\begin{array}{ll}
V_{2n+r}V_{2n+2} & \text{\ if }n\text{ is even} \\
-\Delta U_{2n+r}U_{2n+2} & \text{\ if }n\text{ is odd}%
\end{array}%
\right. .$
\end{lemma}

\begin{proof}[Proofs]
For Lemma \ref{lemma a}, we use Binet's forms of $U_{n}$\ and $V_{n}:U_{n}=%
\frac{\alpha ^{n}-\beta ^{n}}{\alpha -\beta }$\ and $V_{n}=\alpha ^{n}+\beta
^{n}$ where $\alpha $ and $\beta $ are the roots of $x^{2}-px-q=0.$ Lemma %
\ref{lemma c} follows from relations 3. 4. 5. 6. of Lemma \ref{lemma a}. We
obtain Theorem \ref{AA} and Corollary \ref{c1} from relations 3. 4. 5. 6. of
Lemma \ref{lemma a} and Lemma \ref{lemma c}, and Corollary \ref{c2} from
relations (\ref{3.95}), (\ref{3.96}) and 3. 4. of Lemma \ref{lemma a}.
\end{proof}

\section{Applications: extension of \v{C}erin \& Gianella results}

The following Theorem is a generalization of Cerin's Theorems 1.1, 1.2 and
1.3 cited in \cite{cer1}

\begin{theorem}
For all integers $m\geq 0$ and $k$, we have%
\begin{eqnarray}
-p^{2}q+V_{2k}^{2} &=&\Delta U_{2k-1}U_{2k+1}\text{ \ and \ }%
-qV_{2}^{2}+V_{2k-1}^{2}=\Delta U_{2k-3}U_{2k+1},  \label{91} \\
\delta _{n}+A_{n}^{\left( 2k,2k\right) }\left( V\right) &=&\left\{
\begin{array}{ll}
\Delta d_{m}U_{2k+2m+1}U_{2k+2m-1} & \text{\ if }n=2m, \\
-p\Delta c_{m}U_{2k+2m+3}V_{2k+2m-1} & \text{\ if }n=2m+1,%
\end{array}%
\right.  \label{92} \\
\theta _{n}+A_{n}^{\left( 2k-1,2k-1\right) }\left( V\right) &=&\left\{
\begin{array}{ll}
\Delta d_{m}U_{2k+2m-1}^{2} & \text{\ if }n=2m, \\
-p\Delta c_{m}U_{2k+2m+1}V_{2k+2m-1} & \text{\ if }n=2m+1,%
\end{array}%
\right.  \label{93} \\
\xi _{n}+A_{n}^{\left( 2k-1,2k\right) }\left( V\right) &=&\left\{
\begin{array}{ll}
\Delta d_{m}U_{2k+2m-1}U_{2k+2m} & \text{\ if }n=2m, \\
-p\Delta c_{m}U_{2k+2m}V_{2k+2m+1}\ \ \  & \text{\ if }n=2m+1.%
\end{array}%
\right.  \label{94}
\end{eqnarray}

Where $\left( \delta _{n}\right) ,$ $\left( \theta _{n}\right) $ and $\left(
\xi _{n}\right) $ are defined as follows: $\left( \delta _{2m},\delta
_{2m+1}\right) =\left( -qV_{2m+1}^{2},-pq\Delta U_{4m+4}\right) ;\ \ \ $ $%
\left( \theta _{2m},\theta _{2m+1}\right) =\left( -2q\left( 1+d_{m}\right)
,-p^{2}q\Delta c_{m}\right) ;$ $\left( \xi _{2m},\xi _{2m+1}\right) =\left(
-pq\left( 1+d_{m}\right) ,-p\Delta c_{m}\right) .$
\end{theorem}

The relations $\delta _{m}=\delta _{m-2}-p^{2}q\Delta V_{2m}$ and $\theta
_{m}=\theta _{m-2}-p^{2}q\Delta d_{\left( m-1\right) /2}$ for $m$ odd, and $%
\theta _{m}=-\theta _{m-2}-2qV_{m/2}^{2}$ for $m$ even, are easily
established. Then, one verifies that we obtain Theorems of \cite{cer1} when $%
\left( p,q\right) =\left( 1,-1\right) .$

\begin{proof}
For (\ref{91}), we use relation 9. of Lemma \ref{lemma a} with $\left(
n,m,h\right) =\left( 2k,2k-1,1\right) $ and $\left( n,m,h\right) =\left(
2k-1,2k-3,2\right) $ respectively. For relations (\ref{92}), (\ref{93}) and (%
\ref{94}), we use relation (\ref{3.99}) for $\left( r,s\right) =\left(
2k,2k\right) $ resp. $\left( r,s\right) =\left( 2k-1,2k-1\right) $ and $%
\left( r,s\right) =\left( 2k-1,2k\right) $ and noticing that, using
relations 3. 4. 5. 6. of Lemma \ref{lemma a}, we have $\left\{
\begin{array}{l}
U_{4k+4m+2}=U_{2k+2m+3}V_{2k+2m-1}-qU_{4}, \\
V_{4k+4m}=\Delta U_{2k+2m+1}V_{2k+2m-1}+qV_{2}\text{ and }%
V_{4m+2}+2q=V_{2m+1}^{2},%
\end{array}%
\right. $

resp.

$\ \ \ \ \ \ \ \left\{
\begin{array}{l}
U_{4k+4m}=U_{2k+2m+1}V_{2k+2m-1}-qU_{2}, \\
V_{4k+4m-2}=\Delta U_{2k+2m-1}^{2}+2q,%
\end{array}%
\right. ,\ \ \left\{
\begin{array}{l}
U_{4k+4m+1}=U_{2k+2m}V_{2k+2m+1}+1, \\
V_{4k+4m-1}=\Delta U_{2k+2m-1}U_{2k+2m}+pq.%
\end{array}%
\right. $
\end{proof}

\begin{theorem}
For all integers $n\geq 0$ and $r,s,t$ and $k$, the following equalities hold%
\begin{eqnarray}
S_{n}^{\left( s,s+t\right) }\left( V\right) &=&\lambda
_{n}+a_{n+1}V_{s-r}V_{2n+r+s+t},\text{ \ with }\lambda _{n}=-q^{r-s}\Delta
S_{n}^{\left( r,r+t\right) }\left( U\right) ,  \label{5.71} \\
A_{n}^{\left( 2k,2k\right) }\left( V\right) &=&\left\{
\begin{array}{ll}
d_{m}V_{2k+2m}^{2}-2p^{2}\Delta b_{m-1} & \text{\ \ if }n=2m \\
p^{2}\Delta c_{m}\left( 1-a_{k+m+1}V_{2k+2m}\right) & \text{\ \ if }n=2m+1%
\end{array}%
\right. ,  \label{5.81} \\
A_{n}^{\left( 2k+1,2k+1\right) }\left( V\right) &=&\left\{
\begin{array}{ll}
-\Delta U_{2m+1}^{2}+d_{m}V_{2k+2m}V_{2k+2m+2} & \text{\ \ if }n=2m \\
-p^{2}\Delta c_{m}a_{k+m+1}V_{2k+2m+2} & \text{\ \ if }n=2m+1%
\end{array}%
\right. ,  \label{5.91} \\
A_{n}^{\left( 2k,2k+1\right) }\left( V\right) &=&\left\{
\begin{array}{ll}
d_{m}V_{2k+2m}V_{2k+2m+1}-e_{m} & \text{\ \ if }n=2m \\
-p\Delta c_{m}\left( U_{2k+2m+3}V_{2k+2m}-p^{2}+q\right) & \text{\ \ if }%
n=2m+1%
\end{array}%
\right. .  \label{5.101}
\end{eqnarray}
\end{theorem}

\begin{proof}
For (\ref{5.71}) use (\ref{3.13}). For (\ref{5.81}), (\ref{5.91}) and (\ref%
{5.101}), we use (\ref{3.99}) when $r=s=2k$ resp. $r=s=2k+1$ and $\left(
r,s\right) =\left( 2k,2k+1\right) $ using, for $t=0$ resp. $t=2$ and $t=1,$
relations $V_{4k+4m+t}=V_{2k+2m+t}V_{2k+2m}-V_{t}$ and $%
U_{4k+4m+t+2}=U_{2k+2m+2-r\left( r-2\right) }V_{2k+2m+r\left( r-1\right)
}-U_{\left( 2-r\right) \left( 2r+1\right) },$ derived from relations 4. and
5. of Lemma \ref{lemma a}. For (\ref{5.91}), we also use $%
V_{2}d_{m}-2q=V_{4m+2}-2q=\Delta U_{2m+1}^{2}$ derived from 3. of Lemma \ref%
{lemma a}.
\end{proof}

Notice that from the first relation of Theorem \ref{AA}, $\lambda
_{n}=-p^{-1}q^{s-r}\left( U_{4n+2r+t+2}-U_{2r+t-2}\right) +\left( n+1\right)
q^{s}V_{t}$, we have also $e_{m}=pV_{2}^{-1}\left( V_{4m+2}-V_{-2}\right)
=p^{3}\Delta \sum_{j=0}^{m}c_{j-1}$ using first relation of Lemma \ref{lemma
c}$.$

For $\left( p,q\right) =\left( 2,-1\right) ,$ we obtain Theorems 1, 2, 3, 4,
5, 6 and 7 of \v{C}erin and Gianella cited in \cite{cer2}: relation (\ref%
{5.71}), with $\left( s,t\right) =\left( 2k,0\right) $ and $r\in \left\{
0,2,1,-1\right\} $ give respectively Theorem 1 and relations $\left(
2.3\right) ,$ $\left( 2.4\right) $ and $\left( 2.5\right) ,$ with $\left(
s,t\right) =\left( 2k+1,0\right) $ and $r\in \left\{ 2,3\right\} $ give
Theorems 2 and 3, and with $\left( s,t\right) =\left( 2k,1\right) $ and $r=0$
give Theorem 4. Relations (\ref{5.81}), (\ref{5.91}) and (\ref{5.101}) give
Theorems 5, 6 and 7.

Relations 8. 9. of Lemma \ref{lemma a} allow us to obtain immediately the
following Theorem

\begin{theorem}
For all integers $n,\ m,\ r,\ s,$ we have%
\begin{equation*}
V_{n}V_{m}=V_{n+r}V_{m-r}+q^{n}\Delta U_{r}U_{m-n-r}=\Delta
U_{n+s}U_{m-s}+q^{m-s}V_{s}V_{n-m+s}.
\end{equation*}
\end{theorem}

For $\left( p,q\right) =\left( 2,-1\right) ,\ n=2k,\ m=2k+1,\ r=3$ and $s=2,$
and by setting $P_{n}^{\star }=2P_{n}$ for all $n,$ one gets $%
Q_{2k}Q_{2k+1}=Q_{2k+3}Q_{2k-2}-80=8P_{2k+2}P_{2k-1}-12=2\left(
P_{2k+2}^{\star }P_{2k-1}^{\star }-6\right) $ which is Theorem 8 of \cite%
{cer2}, where \v{C}erin and Gianella called $\left( P_{n}^{\star }\right)
_{n}$ the Pell sequence instead of $\left( P_{n}\right) _{n}.$

\end{document}